\documentclass[11pt]{amsart}
\usepackage{amsmath,amssymb}
\usepackage[alphabetic]{amsrefs}
\usepackage{mathrsfs}
\usepackage{graphicx,color}
\usepackage{wrapfig}
\usepackage{calc}
\usepackage{extarrows}
\usepackage{amsthm}
\usepackage{latexsym}
\usepackage{slashed}
\usepackage{framed}
\usepackage{ascmac}
\usepackage[all]{xy}
\usepackage{hyperref}
\usepackage[mathscr]{eucal}

\BibSpec{book}{%
    +{}  {\PrintPrimary}                {transition}
    +{.} { \textit}                     {title}
    +{.} { }                            {part}
    +{:} { \textit}                     {subtitle}
    +{,} { \PrintEdition}               {edition}
    +{}  { \PrintEditorsB}              {editor}
    +{,} { \PrintTranslatorsC}          {translator}
    +{,} { \PrintContributions}         {contribution}
    +{,} { }                            {series}
    +{,} { \voltext}                    {volume}
    +{,} { }                            {publisher}
    +{,} { }                            {organization}
    +{,} { }                            {address}
    +{,} { }                            {status}
    +{,} { \PrintDOI}                   {doi}
    +{,} { \PrintISBNs}                 {isbn}
    +{}  { \parenthesize}               {language}
    +{}  { \PrintTranslation}           {translation}
    +{;} { \PrintReprint}               {reprint}
    +{,} { \PrintDate}                  {date}
    +{.} { }                            {note}
    +{.} {}                             {transition}
}
\BibSpec{article}{%
    +{}  {\PrintAuthors}                {author}
    +{,} { \textit}                     {title}
    +{.} { }                            {part}
    +{:} { \textit}                     {subtitle}
    +{,} { \PrintContributions}         {contribution}
    +{.} { \PrintPartials}              {partial}
    +{,} { }                            {journal}
    +{}  { \textbf}                     {volume}
    +{}  { \PrintDatePV}                {date}
    +{,} { \issuetext}                  {number}
    +{,} { \eprintpages}                {pages}
    +{,} { }                            {status}
    +{,} { \PrintDOI}                   {doi}
    +{}  { \parenthesize}               {language}
    +{}  { \PrintTranslation}           {translation}
    +{;} { \PrintReprint}               {reprint}
    +{.} { }                            {note}
    +{,} { \eprint}                     {eprint}
    +{.} {}                             {transition}
}
\BibSpec{collection.article}{%
    +{}  {\PrintAuthors}                {author}
    +{,} { \textit}                     {title}
    +{.} { }                            {part}
    +{:} { \textit}                     {subtitle}
    +{,} { \PrintContributions}         {contribution}
    +{,} { \PrintConference}            {conference}
    +{}  {\PrintBook}                   {book}
    +{,} { }                            {booktitle}
    +{,} { \PrintDateB}                 {date}
    +{,} { pp.~}                        {pages}
    +{,} { }                            {publisher}
    +{,} { }                            {organization}
    +{,} { }                            {address}
    +{,} { }                            {status}
    +{,} { \PrintDOI}                   {doi}
    +{,} { \eprint}        {eprint}
    +{}  { \parenthesize}               {language}
    +{}  { \PrintTranslation}           {translation}
    +{;} { \PrintReprint}               {reprint}
    +{.} { }                            {note}
    +{.} {}                             {transition}
}

\BibSpec{misc}{
  +{}{\PrintAuthors}  {author}
  +{,}{ \textit}      {title}
  +{,}{ }             {note}
  +{,}{ }             {date}
  +{.}{}              {transition}
}

\makeatletter
\@addtoreset{equation}{section}
\makeatother

\DeclareMathAlphabet{\mathfrak}{U}{euf}{m}{n}
\SetMathAlphabet{\mathfrak}{bold}{U}{euf}{b}{n}

\newcommand{\ma}[1]{\begin{align*} #1 \end{align*}}
\newcommand{\maa}[1]{\begin{align} #1 \end{align}}

\theoremstyle{definition}
\newtheorem{defn}[equation]{Definition}
\theoremstyle{plain}
\newtheorem{thm}[equation]{Theorem}
\newtheorem{prp}[equation]{Proposition}
\newtheorem{lem}[equation]{Lemma}
\newtheorem{cor}[equation]{Corollary}

\theoremstyle{remark}
\newtheorem{remk}[equation]{Remark}
\newtheorem{exmp}[equation]{Example}
\newtheorem*{remk*}{Remark}
\newtheorem{blank}[equation]{}

\newcommand{\real}{\mathbb R}
\newcommand{\comp}{\mathbb C}
\newcommand{\zahl}{\mathbb Z}
\newcommand{\quot}{\mathbb Q}
\newcommand{\nat}{\mathbb N}

\newcommand{\id}{\mathrm{id}}
\DeclareMathOperator{\rank}{rank}

\newcommand{\ssbk}[1]{\left\| #1 \right\|}

\newcommand{\pmx}[1]{\begin{pmatrix} #1 \end{pmatrix}}

\newcommand{\dalpha}{\hat \alpha}

\DeclareMathOperator{\Res}{Res}
\newcommand{\cone}{\mathcal{C}}

\newcommand{\J}{\mathfrak{J}}
\newcommand{\K}{\mathrm{K}}
\newcommand{\KK}{\mathrm{KK}}

\newcommand{\Cst}{\mathrm{C}^*}
\newcommand{\Kop}{\mathbb{K}}
\newcommand{\Bop}{\mathbb{B}}
\newcommand{\Lop}{\mathbb{L}}

\newcommand{\Hilb}{\mathscr{H}}

\newcommand{\Rdim}{\dim_{\mathrm{cRok}}}

\DeclareMathOperator{\Ad}{Ad}
\newcommand{\cO}{\mathcal{O}}
\newcommand{\dG}{\hat{G}}

\label{thm:Rokinj}

\newcommand{\Irr}{\operatorname{\rm Irr}}

\newcommand{\dD}{\hat \Delta}

\title[Compact Lie group actions with continuous Rokhlin property]{Compact Lie group actions with continuous Rokhlin property}
\date{\today}

\author[Y. Arano]{Yuki Arano}
\address{Graduate School of Mathematical Science, The University of Tokyo, 3-8-1 Komaba, Meguro-ku, Tokyo 153-8914, Japan}
\email{arano@ms.u-tokyo.ac.jp}

\author[Y. Kubota]{Yosuke Kubota}
\address{Graduate School of Mathematical Science, The University of Tokyo, 3-8-1 Komaba, Meguro-ku, Tokyo 153-8914, Japan}
\email{ykubota@ms.u-tokyo.ac.jp}

\begin{document}
\maketitle
\begin{abstract}
In this paper, we study continuous Rokhlin property of $\mathrm{C}^*$-dynamical systems using techniques of equivariant $\mathrm{KK}$-theory and quantum group theory. In particular, we determine the $\mathrm{KK}$-equivalence class and give a classification of Kirchberg $G$-algebras when the $G$ is a compact Lie group with Hodgkin condition. 
\end{abstract}

\section{Introduction}
After the initial work by Connes~\cite{MR0388117}, the study of group actions on $\Cst$-algebras and von Neumann algebras, particularly their classification, is a fundamental subject in the theory of operator algebras, as well as the classification of operator algebras themselves.
In this paper, we focus on actions of compact Lie groups on $\Cst$-algebras. 

In the study of von Neumann algebras, the classification of compact group actions on factors are studied in Kawahigashi-Takesaki~\cite{MR1156672}, Masuda-Tomatsu~\cite{MR2578461} and so on. 
In the context of $\Cst$-algebras, Izumi~\citelist{\cite{MR2053753}\cite{MR2047851}}  introduced Rokhlin property of finite group actions on $\Cst$-algebras extracting an essential property which is used to classify group actions in the von Neumann algebra theory and classified such actions.
He also shows that many of fundamental $\Cst$-dynamical systems such as infinite tensor products of left regular actions of a finite group $G$ on $\mathbb{M} _{|G|}$ and the quasi-free action of a finite group $G$ on $\cO_{|G|}$ by a regular representation has this property.
After his work, the study of this kind of actions has been attracting attention among $\Cst$-algebraists~\citelist{\cite{MR2808327}\cite{MR2875821}\cite{MR3316642}\cite{MR3398712}}.
Among those, recently, Gardella~\cite{mathOA14081946} initiated the study of Rokhlin actions for general compact groups.

Roughly speaking, Rokhlin property is an analogue in an approximate sense of freeness of product type (that is, $G$-actions on $G \times X$) in topological dynamics.
For a $G$-$\Cst$-algebra $A$ with Rokhlin property, there is a fundamental technique replacing a projection or a unitary with another $G$-invariant one, which is called an \emph{averaging technique} by Gardella~\cite{mathOA14075485}. More precisely, there is a sequence of completely positive maps $A \to A^\alpha$ which is approximately a $\ast$-homomorphism preserving $A^\alpha$. 
This gives a strong restriction on the structure of $\Cst$-dynamical systems, its $\K$-groups, crossed products and so on. 
For example, we can prove certain ``approximate cohomology vanishing" type theorems, which play an essential role in the study of actions on von Neumann algebras.

Our main tool is Kasparov's $\KK$-theory~\cite{MR582160}. 
Although it is a kind of (co)homology theory and hence does not distinguish two homotopic $\Cst$-algebras, $\KK$-theory is also a powerful tool for classifications of $\Cst$-algebras up to isomorphism. 
For example, the Kirchberg-Phillips classification \citelist{\cite{Kirchberg}\cite{MR1745197}} asserts that two Kirchberg algebras are isomorphic if and only if there is a $\KK$-equivalence preserving the unit classes in $K_0$-groups. 
Here, it is essential in the proof of \cite{MR1745197} that the $\KK$-group has a presentation as the set of homotopy classes of asymptotic morphisms (cf.\ $E$-theory \cite{MR1065438}).

In \cite{mathOA150806815}, the authors study the relative homological algebra of compact group equivariant $\KK$-theory in connection with the Atiyah-Segal completion theorem. 
Here, we introduce a categorical counterpart of freeness of group actions on $\Cst$-algebras, $\J_G^n$-injectivity. 
It directly follows from Thomsen's picture of equivariant $\KK$-theory \cite{MR1680467} that $G$-$\Cst$-algebras with continuous Rokhlin property \cite{mathOA14071277}, a variation of Rokhlin property, is $\J_G$-injective. 
More generally, if a $\Cst$-dynamical system has continuous Rokhlin dimension with commuting tower \citelist{\cite{mathOA12091618}\cite{mathOA14071277}} less than $d$, then it is $\J_G^{d+1}$-injective (Theorem \ref{thm:Rokinj}). 
This theorem is also available for the study of Rokhlin property because the Rokhlin dimension with commuting tower is finite if and only if so is the continuous Rokhlin dimension with commuting tower (Proposition \ref{prp:RokcRok}).

Although continuous Rokhlin property is strictly stronger than Rokhlin property, many of known examples of Rokhlin actions have continuous Rokhlin property. Actually, a necessary and sufficient condition for a Rokhlin action on a Kirchberg algebra to have continuous Rokhlin property is given in terms of equivariant $\KK$-theory (Proposition \ref{prp:injcRok}). 

Our first main theorem (Theorem \ref{thm:KKHod}) asserts that any $G$-$\Cst$-algebra with continuous Rokhlin property is $\KK^G$-equivalent to $A^\alpha \otimes C(G)$ when $G$ is a compact Lie group with Hodgkin condition (connected and $\pi _1(G)$ is torsion-free). 
To see this, we use the strong Baum-Connes isomorphism for an arbitrary coefficient of the dual quantum group $\dG$ of a Hodgkin Lie group $G$ \cite{MR2339371}, which is rephrased in terms of the crossed product functor since $\dG$ is torsion-free in the sense of \cite{MR2563811}.

Next, we apply it for classification of Kirchberg $G$-algebras with continuous Rokhlin property. 
In \cite{MR2053753}, Izumi gives a necessary and sufficient condition for two actions with Rokhlin property to be conjugate by using an intertwining argument. 
By the same argument, we prove that two Kirchberg $G$-algebras are $G$-equivariantly isomorphic if and only if they are $\KK^G$-equivalent. 
Together with Theorem \ref{thm:KKHod}, we obtain a complete classification of Kirchberg $G$-algebras (Theorem \ref{thm:Kir}) in terms of equivariant $\KK$-theory, which is our second main theorem. 
An essential part in Theorem \ref{thm:Kir} is construction of the model action $\cO(G)$ of Kirchberg $G$-algebra with continuous Rokhlin property which is $\KK^G$-equivalent to $C(G)$. We construct it as a crossed product of the dual discrete quantum group $\hat{G}$ by an asymptotically representable action on $\cO_\infty$. We use Kishimoto's argument~\cite{MR634163} in order to show that the crossed product is again a Kirchberg algebra.

\section{Continuous Rokhlin property for compact group actions}\label{section:2}
Throughout this paper, we assume that all topological groups are second countable Hausdorff and all $\Cst$-algebras are unital and separable unless otherwise noted.
\subsection{Definition and Examples}
Let $G$ be a compact group and let $(A,\alpha)$ be a unital $G$-$\Cst$-algebra. Set
\ma{
\mathfrak{T}_\alpha A&:=\{ a=(a_t) \in C_b([0,\infty),A) \mid \text{$(a_t)$ is uniformly $G$-continuous} \}, \\
\mathfrak{A}_\alpha A &:=\mathfrak{T}_\alpha A /C_0([0,\infty ) ,A), \ \mathfrak{C}_\alpha A:=\mathfrak{A}_\alpha A \cap A'.
}
They are equipped with the canonical $G$-$\Cst$-algebra structure. We say the (nonseparable) $\Cst$-algebra $\mathfrak{C}_\alpha A$ is the \emph{central path algebra} of $(A,\alpha)$.

\begin{defn}
Let $G$ be a compact group and let $X$ be a compact $G$-space. 
\begin{enumerate}
\item We say that $(A,\alpha)$ has \emph{continuous Rokhlin property} if there is an $G$-equivariant unital $\ast$-homomorphism $\varphi : C(G) \to \mathfrak{C}_\alpha A$.
\item We say that $(A,\alpha)$ has \emph{continuous $X$-Rokhlin property} if there is an $G$-equivariant unital $\ast$-homomorphism $\varphi :C(X) \to \mathfrak{C}_\alpha A$.
\item We say that $(A,\alpha)$ has \emph{continuous Rokhlin dimension} less than $d$ \emph{with commuting tower} and write as $\Rdim ^c (A,\alpha)\leq d$ if there are completely positive contractive order zero maps $\varphi ^0, \dots , \varphi ^d: C(G) \to \mathfrak{C}_\alpha A$ with commuting ranges such that $\varphi ^0(1)+\dots +\varphi ^d(1)=1$. 
\end{enumerate}
Here we say that this $\varphi$ is a continuous Rokhlin map for $(A,\alpha)$. 
\end{defn}
These are variations of Rokhlin property (Definition 3.1 of \cite{MR2053753}) and the Rokhlin dimension (Definition 2.3 of \cite{mathOA12091618}), which are defined by using $\zahl _{>0}$ instead of $[0,\infty)$.

A $G$-$\Cst$-algebra $A$ has continuous Rokhlin dimension less than $d-1$ with commuting tower if and only if $A$ has continuous $E_dG$-Rokhlin property where $E_dG:=G \ast \cdots \ast G$ is the $d$-th step of the Milnor construction (cf.\ Lemma 1.7 of \cite{MR3398712} and Lemma 4.4 of \cite{mathOA14071277}). To see this, remind that a completely positive contractive map of order zero from $A$ to $B$ is of the form $\varphi (ta)$ where $\varphi$ is a unital $\ast$-homomorphism from $\cone A:= C_0((0,1], A)^+$ to $B$  (Corollary 4.1 of \cite{MR2545617}). We remark that $\varphi (ta)$ is $G$-equivariant if and only if so does $\varphi$ and the join $X \ast Y$ of two compact spaces satisfies $\cone (X \ast Y) \cong \cone X \times \cone Y$ where $\cone X:= X \times [0,1]/{X \times \{ 0 \}}$. 

\begin{prp}\label{prp:RokcRok}
For a $G$-$\Cst$-algebra, we have
\[\dim _{\mathrm{Rok}}^c (A,\alpha) \leq \Rdim ^c (A, \alpha ) \leq 2 \dim _{\mathrm{Rok}}^c (A,\alpha) +1. \]
In particular, we have $\Rdim ^c(A,\alpha ) \leq 1$ if $A$ has Rokhlin property. 
\end{prp}
\begin{proof}
We write $A_{\alpha, \infty}:=\ell ^\infty_\alpha (\nat, A) /c_0(\nat , A) \cap A'$ as in \cite{mathOA14081946} and let $\varphi^i: C(G) \to A_{\alpha ,\infty}$ be completely positive maps of order zero such that $\varphi ^0(1)+\cdots +\varphi ^d(1)=1$. Let $F_n$ be an increasing sequence of finite subsets of $C(G)$ such that $\overline{\bigcup F_n}=C(G)$ and choose $\varphi _n$ inductively such that $\ssbk{[\varphi ^i_n(f), \varphi ^i_m(f)]} \leq 2^{-n-m}$ for any $n,m \geq N$, $f \in F_N$ and $i=0,\dots, d$.

Then, we obtain $(2d+2)$ $G$-equivariant completely positive maps from $C(G)$ to $\mathfrak{C}_\alpha A$ of order zero given by
\[
\tilde{\varphi} _{2n+j+t}^{i,j}(f)=(1-|t|)\varphi ^i_n(f) \ \  \text{ for $t \in [-1,1]$}
\]
for $i=0,\dots ,d$ and $j=0,1$. By definition, the images in $\mathfrak{C}_\alpha A$ commute and $\sum \tilde{\varphi}^{i,j}_t(1)=1$.
\end{proof}

\begin{remk}
Continuous Rokhlin property is actually strictly stronger than Rokhlin property. See Subsection \ref{section:5.3} for more detail.
\end{remk}

\begin{exmp}
The following examples are pointed out by Eusebio Gardella. Let $G$ be a finite group. The UHF algebra $\mathbb{M}_{|G|^\infty}$ with the $G$-action $\alpha _g :=\bigotimes ^\infty \Ad \lambda _g$ has continuous Rokhlin property. Actually, we obtain an asymptotically central path of mutually orthogonal projections $\{ p_t^g \}_{g \in G}$ satisfying $\alpha _h(p_t^g)=p^{hg}_t$ given by
$$p_t^g:=1 \otimes \cdots \otimes 1 \otimes u_{t-n} (p_g \otimes 1) u_{t-n}^* \otimes 1 \otimes  \cdots $$
for $t \in [n,n+1]$. Here $p_g \in \mathbb{M}_{|G|}$ is the projection corresponding to $g \in G$ and $u_t$ is a homotopy of $G$-invariant unitaries in $\mathbb{M}_{|G|}^{\otimes 2}$ such that $u_0=1$ and $u_1$ is the flip on $\comp ^{|G|} \otimes \comp ^{|G|}$.

Since $\cO_n \otimes \mathbb{M}_{n^\infty} \cong \cO_n$, we obtain an example of continuous Rokhlin actions on the Cuntz algebras $\mathcal{O}_{|G|}$. On the other hand, since any automorphism on $\cO_{|G|}$ is approximately inner (Theorem 3.6 of \cite{MR1225963}), Theorem 3.5 of \cite{MR2053753} implies that every Rokhlin action on $\cO_{|G|}$ is conjugate to the above action and hence has continuous Rokhlin property. In particular, the quasi-free action with respect to the left regular representation has continuous Rokhlin property (Proposition 5.6 of \cite{MR2053753}). Similarly, by the above example and Kirchberg's absorption theorem (Theorem 3.2 of \cite{MR1780426}), the unique Rokhlin action of $G$ on $\mathcal{O}_2$ (Theorem 4.2 of \cite{MR2053753}) has continuous Rokhlin property.
\end{exmp}

\begin{exmp}
Pick $\theta, \omega \in \real \setminus \quot $ such that $\theta-\omega \in \real \setminus \quot$ and let
\[
A:=\zahl {\, }_{\alpha _{\theta}} \!\! \ltimes(C(\mathbb{T}) \rtimes _{\alpha _\omega}\zahl)
\] 
be a noncommutative torus where we write $\alpha _z$ for the rotation automorphism on $C(\mathbb{T})$ (we use the same letter $\alpha _z$ for the automorphism $\alpha _t \rtimes \zahl$ on $C(\mathbb{T}) \rtimes \zahl$). Then, $A$ is simple by Theorem 1.9 of \cite{mathOA0609783} and the dual action $\hat{\alpha}_{\theta}$ of $\mathbb{T}$ has Rokhlin property since $\alpha _{\theta}$ is approximately representable (cf.\ Lemma \ref{lem:asymprep}). Moreover, Proposition \ref{prp:injcRok} implies that $(A \otimes \cO_\infty, \alpha \otimes \id )$ has continuous Rokhlin property (note that $A$ is $\KK^{\mathbb{T}}$-equivalent to $C(\mathbb{T}^3)$ and hence $\J_{\mathbb{T}}$-injective). 
\end{exmp}

\subsection{Averaging technique via equivariant $\KK$-theory}\label{section:3}
A fundamental technique for $\Cst$-dynamical systems with (continuous) Rokhlin property is an averaging process. 
Let $A$ be a separable $G$-$\Cst$-algebra with continuous $X$-Rokhlin property. 
We use the same letter $\varphi$ for its $G$-equivariant completely positive contractive lift, which is a $G$-equivariant asymptotic morphism from $C(X)$ to $A$. 
Let us choose an increasing sequence $F_n$ of finite subsets of $A$ such that $\overline{\bigcup F_n}=A$ and families $\{ f_{n,i} \}_{i \in I_n}$ of positive continuous functions on $X$ as Lemma 4.2 of \cite{mathOA14075485}.
Then, we obtain a $G$-equivariant completely positive asymptotic morphism $\xymatrix@C=1em{\psi :A \otimes C(X) \ar@{.>}[r] &A}$ defined by 
\maa{
\psi _t(\xi ):=&(t-n) \sum _{i}\varphi _{\chi(t)}(f_{n,i})^{1/2}\xi(x_{n,i}) \varphi _{\chi (t)} (f_{n,i})^{1/2} \label{form:ab}\\
&+(n+1-t) \sum _i \varphi _{\chi (t)}(f_{n+1,i})^{1/2}\xi (x_{n+1,i}) \varphi _{\chi (t)} (f_{n+1,i})^{1/2} \notag }
for $t \in [n,n+1]$ where $\chi : \real _{\geq 0} \to \real _{\geq 0}$ is a homeomorphism such that $\ssbk{[\varphi_t(f_n) ,a] } <2^{-n} |I_n|^{-1}$ for any $t \geq \chi (n)$ and $a \in F_n$. 
Note that $\psi =\id _A \times \varphi _\chi$ as a $\ast$-homomorphism from $A \otimes C(X)$ to $\mathfrak{A}_\alpha A$.

This averaging map is compatible with the picture of $\KK$-theory given in \citelist{\cite{MR1673935}\cite{MR1680467}} using completely positive asymptotic morphisms. Actually we have the isomorphism
\maa{\KK ^G(A,B) \cong [\![ SA \otimes \Kop _G, SB \otimes \Kop _G ]\!] _{\mathrm{cp}}^G \label{form:KKcp}}
where $[\![ A, B ]\!] _{\mathrm{cp}}^G$ is the set of homotopy classes of completely positive $G$-equivariant asymptotic morphisms from $A$ to $B$. Moreover the Kasparov product is given by the composition of asymptotic morphisms. Hence $\psi$ gives an element $\KK ^G(A \otimes C(X),A)$ such that $[\psi] \circ \iota =\id _A$. As an immediate consequence, we obtain the following.

\begin{thm}\label{thm:Rokinj}
Let $A$ be a separable unital $G$-$\Cst$-algebra. If $A$ has continuous Rokhlin property, then $A$ is $\J_G$-injective. Moreover, if $A$ has continuous Rokhlin dimension with commuting tower less than $d-1$, then $A$ is $\J ^{d}_G$-injective.
\end{thm}
Here, we say that a separable $G$-$\Cst$-algebra $A$ is $\J_G^d$-injective if the $\KK^G$-morphism induced from the canonical inclusion $A \to A \otimes C(E_dG)$ has a left inverse (see \cite{mathOA150806815} for more details). 

\begin{cor}\label{cor:equiv}
Let $A$ and $B$ be $G$-$\Cst$-algebras with finite (continuous) Rokhlin dimension with commuting tower. For $\phi \in \KK^G(A, B)$, $\phi$ is a $\KK^G$-equivalence if and only if $\Res _G \phi$ is a $\KK$-equivalence. Moreover, if $A$ and $B$ are in the UCT class, $\phi$ is a $\KK^G$-equivalence if and only if $(\Res _G\phi) _*:\K_*(A) \to \K_*(B)$ is an isomorphism.
\end{cor}
\begin{proof}
It follows from Theorem \ref{thm:Rokinj}, Proposition \ref{prp:RokcRok}, Lemma \ref{lem:avbdl}, Theorem 3.4 of \cite{mathOA150806815} and the universal coefficient theorem (Proposition 7.3 of \cite{MR894590}).
\end{proof}

\begin{remk}
In fact, $\J_G$-injectivity is more relevant to asymptotic innerness of the dual action rather than asymptotic representability, which we need for continuous Rokhlin property (see Remark \ref{rmk:inner}).
For example, the crossed product of the infinite tensor product on a $\Cst$-algebra by the Bernoulli shift of $\zahl$ is a $\J_\mathbb{T}$-injective $\mathbb{T}$-$\Cst$-algebra without continuous Rokhlin property.
\end{remk}

\begin{lem}\label{lem:avbdl}
Let $A$ be a $G$-$\Cst$-algebra with Rokhlin property. Then, there is an isomorphism $G \ltimes A \cong A^\alpha \otimes \Kop$. In particular, if $A$ and $B$ are $G$-$\Cst$-algebras with Rokhlin property which are $\KK^G$-equivalent, then $A^\alpha$ and $B^\beta$ are $\KK$-equivalent.
\end{lem}
\begin{proof}
It suffices to show that $A \otimes \Kop(L^2(G))\cong A \otimes \Kop$ because $(A \otimes \Kop)^G=A^\alpha \otimes \Kop$ and $(A \otimes \Kop (L^2(G)))^G=G \ltimes A$. 
To this end, we show that for any two unitary representations $\Hilb _0$ and $\Hilb _1$ with the same (finite) dimension, $\Hilb _0 \otimes A \cong \Hilb_1 \otimes A$ as $G$-equivariant Hilbert $A$-modules. 

Let $u_0$ be an equivariant isomorphism $\Hilb _0 \otimes C(G)\to \Hilb _1 \otimes C(G)$ and set $u:=\pmx{0 & u_0^* \\ u_0 & 0} \in \Bop (\Hilb) \otimes C(G)$ where $\Hilb :=\Hilb _0 \oplus \Hilb _1$. 
Since $A$ has Rokhlin property, there is a $G$-equivariant completely positive map $\varphi ':= \varphi \otimes \id _{\Bop (\Hilb)} :C(G)\otimes \Bop (\Hilb ) \to A \otimes \Bop (\Hilb)$ such that $\ssbk{\varphi'(u)^*\varphi'(u)-1}<\varepsilon$ and $\ssbk{\varphi'(u)\varphi'(u)^*-1}<\varepsilon$. Now $\tilde{u}:=\varphi' (u)|\varphi' (u)|^{-1/2}$ is a $G$-invariant odd unitary on $\Hilb \otimes A$ which induces the $G$-equivariant isomorphism from $\Hilb _0 \otimes A$ to $\Hilb _1 \otimes A$.
\end{proof}

\begin{cor}
Let $A$ be a $\Cst$-algebra and let $\{ \alpha _t\}$ be a homotopy of $G$-actions on $A$ such that $\alpha _0$ and $\alpha _1$ has finite (continuous) Rokhlin dimension with commuting tower. Then $(A,\alpha _0)$ and $(A, \alpha _1)$ are $\KK^G$-equivalent. Moreover, if $\alpha _0$ and $\alpha _1$ have Rokhlin property, the fixed point subalgebras $A^{\alpha_0}$ and $A^{\alpha _1}$ are $\KK$-equivalent.
\end{cor}
\begin{proof}
It follows from Theorem \ref{thm:Rokinj}, Proposition \ref{prp:RokcRok}, Lemma \ref{lem:avbdl} and Corollary 3.5 of \cite{mathOA150806815}.
\end{proof}

\section{Continuous Rokhlin action of Hodgkin Lie groups}
In this section, we focus on the case that $G$ is a compact Lie group with Hodgkin condition, that is, $G$ is connected and $\pi _1(G)$ is torsion-free. 
An important feature of Hodgkin Lie groups is that the dual quantum group $\hat{G}$ is torsion-free in the sense of Section 7.2 of \cite{MR2563811}. Together with the strong Baum-Connes conjecture for $\dG$ (Corollary 3.4 of  \cite{MR2339371}), we obtain the following.
\begin{prp}\label{prp:BC}
Let $G$ be a Hodgkin Lie group. A $\KK^G$-morphism $\xi \in \KK^G(A,B)$ is a $\KK ^G$-equivalence if and only if $G \ltimes \xi \in KK(G \ltimes A, G \ltimes B)$ is a $\KK$-equivalence. 
\end{prp}

\begin{thm}\label{thm:KKHod}
Let $G$ be a compact Lie group with Hodgkin condition. Then, for any $G$-algebra $A$ with continuous Rokhlin property, there is a $\KK^G$-equivalence from $A$ to $C(G) \otimes A^\alpha$ mapping $[1_A] \in \K^G_0(A)$ to $[1_{A^\alpha \otimes C(G)}]$.
\end{thm} 
\begin{proof}
By Proposition \ref{prp:BC}, it suffices to show that $G \ltimes \xi$ is a $\KK$-equivalence for 
\[
\xi:=[\psi |_{A^\alpha \otimes C(G)}] \in \KK^G(A^\alpha \otimes C(G), A).
\]

Let $p$ be the projection in $C^*_\lambda(G)$ corresponding to the trivial representation and consider the canonical inclusions $\iota_1: A^\alpha \to A^\alpha \otimes C(G)$, $\iota_2 : A^\alpha \to A$ and $\mu: \comp p \otimes A^\alpha \to C^*_\lambda(G) \otimes A^\alpha$. Since the restriction of $\psi$ to $A^\alpha $ is the identity, we have $\xi \circ [\iota_1]=[\iota_2]$ and hence $(G \ltimes \xi) \circ (G \ltimes [\iota_1])=G \ltimes [\iota_2]$. We also have $G \ltimes [\iota_k] \circ [\mu]$ are $\KK$-equivalences for $k=1,2$ since the isomorphism in Lemma \ref{lem:avbdl} maps $p$ to a $1$-dimensional projection in $\Kop$. Hence $G \ltimes \xi=(G \ltimes [i_2] \circ [\mu])(G \ltimes [i_1] \circ [\mu])^{-1}$ is a $\KK$-equivalence.
\end{proof}

\begin{remk}\label{rmk:inner}
In fact, $\KK^G$-equivalence $A \sim B \otimes C(G)$ for some $B$ holds under weaker assumption for the action. As in \cite{mathOA11072512}, we can apply the Baum-Connes conjecture for the dual of Hodgkin Lie groups to the path of cocycle actions using the duality \cite{MR1970242} and we obtain any $G$-$\Cst$-algebra is $\KK^G$-equivalent to $C(G) \otimes B$ for some $B$ if its dual action is asymptotically inner.
\end{remk}

The following corollary is an analogue of Theorem 5.5 of \cite{mathOA14052469} for $\mathbb{T}$-$\Cst$-dynamical systems.

\begin{cor}\label{cor:Krid}
Let $G$ be a Hodgkin Lie group and let $A$ be a $G$-$\Cst$-algebra with continuous Rokhlin property. Then, there is a countable abelian group $M$ such that $\K$-groups $\K_*(A)$ is isomorphic to $M _*^n$ where $n:=2^{\rank G -1}$. Moreover, in this case $M \cong \K_0(A^\alpha) \oplus \K_1(A^\alpha)$.
\end{cor}
\begin{proof}
It follows from Theorem \ref{thm:KKHod}. We remark that $C(G)$ is $\KK$-equivalent to $\comp ^{n} \oplus C_0(\real)^n$, which follows from the fact that $\K^*(G)$ ($\ast =0,1$) are torsion-free abelian groups of rank $n$ (Theorem A (i) of \cite{MR0214099} and Hopf's theorem, see for example Theorem 1.34 of \cite{MR2403898}) and the universal coefficient theorem \cite{MR894590}.  
\end{proof}

\begin{cor}\label{cor:classify}
Let $G$ be a Hodgkin Lie group. Two $G$-$\Cst$-algebras with continuous Rokhlin property are $\KK^{G}$-equivalent  if and only if their fixed point algebras are $\KK$-equivalent. In particular, when these $\Cst$-algebras are in the UCT-class, then they are $\KK^{G}$-equivalent if and only if the $\K _*$-groups of fixed point algebras are isomorphic.
\end{cor}
\begin{proof}
It follows from Lemma \ref{lem:avbdl} and Theorem \ref{thm:KKHod}.
Note that $A$ is UCT if and only if so is $A^\alpha $ because $C(G)$ is $\KK$-equivalent to $\comp ^{n} \oplus C_0(\real)^n$.  
\end{proof}

\section{Kirchberg $G$-algebras with continuous Rokhlin property}\label{section:4}
In this section, we give a complete classification of Kirchberg $G$-algebras with continuous Rokhlin property up to conjugacy in terms of (equivariant) $\KK$-theory.

\subsection{Preliminaries on duals of compact groups}
First, we collect some terminologies around the dual of compact groups, which are also used in Proposition \ref{prp:BC} implicitly. We refer to \cite{MR3204665} for more information.
Let $G$ be a compact group. Then, the group $\Cst$-algebra $C^*_\lambda(G)$ together with the nondegenerate $*$-homomorphism
\[\dD: C^*_\lambda(G) \to M(C^*_\lambda(G) \otimes C^*_\lambda(G))\]
determined by $\dD(\lambda_g) = \lambda _g \otimes \lambda _g$ for $g \in G$ is called the dual discrete quantum group $\dG$ of $G$ (Definition 1.6.7 of \cite{MR3204665}). 

A \emph{representation} of $\dG$ (or a corepresentation of $G$) on a $\Cst$-algebra $A$ is a unitary $U \in M(C^*_\lambda(G) \otimes A)$ such that
\[(\dD \otimes \id)(U) = U_{13} U_{23}.\]
The \emph{regular representation} is a representation $W \in M(C^*_\lambda(G) \otimes C(G))$ of $\dG$ on $C(G)$ given by $W(g) :=\lambda_g^*$ regarding $M(C^*_\lambda(G) \otimes C(G))$ as the $\Cst$-algebra of $M(C^*_\lambda(G))$-valued strictly continuous functions on $G$.

There is a natural one-to-one correspondence between representations of $\dG$ on $A$ and unital $\ast$-homomorphisms $C(G) \to M(A)$ as following:
\begin{itemize}
\item For a representation $U$ of $\dG$ on $A$, the representation $\pi_U(f) := (f \otimes \id)(U) \in M(A)$
 of  $B(G) = C^*_\lambda(G)^* \subset C(G)$ extends to a $\ast$-homomorphism of $C(G)$.
\item For a unital $\ast$-homomorphism $\pi$ of $C(G)$ on $M(A)$, then
$U_\pi := (\id \otimes \pi)(W)$
is a representation of $\dG$.
\end{itemize}

A (left) \emph{action} of $\dG$ on a $\Cst$-algebra $A$ is a nondegenerate $\ast$-homomorphism $\alpha : A \to  M(C^*_\lambda(G) \otimes A)$ such that
\[
(\id \otimes \alpha)\alpha = (\dD \otimes \id) \alpha .
\]
For example, for any $G$-$\Cst$-algebra $A$, the crossed product $G \ltimes A$ is a $\dG$-$\Cst$-algebra by the $*$-homomorphism $\dalpha: G \ltimes A \to M(C^*_\lambda(G) \otimes (G \ltimes A))$
uniquely determined by $\dalpha(a) = a$ and $\dalpha(u_g) = u_g \otimes u_g$ for $a \in A$, $g \in G$.

For an action $\alpha$ of $\dG$ on $A$, the reduced \emph{crossed product} is defined by 
\[
\dG {\, }_{\alpha}\! \ltimes A := \overline{\mathrm{span}} (C(G) \otimes 1) \alpha(A) \subset \Lop(L^2(G) \otimes A).
\]
We often omit $\alpha$ when no confusion arise. We also identify $A$ and $C(G)$ as subalgebras of $\dG \ltimes A$.
It is equipped with the dual $G$-action uniquely determined by $\dalpha _g(a)= a$ for $a \in A$ and $\dalpha _g(f)(x)= f(g^{-1} x) \in C(G)$ for $f \in C(G)$. The Baaj-Skandalis duality \cite{MR1235438} asserts that for any $\dG$-$\Cst$-algebra $A$, $G {\, }_{\dalpha}\! \ltimes (\dG {\, }_{\alpha}\! \ltimes  A)$ is isomorphic to $\Kop(L^2(G)) \otimes A$ with the $\dG$-action
\[\tilde \alpha(a) = W_{12}^* (\id \otimes \alpha)(a)_{213} W_{12} \]
 as $G$-$\Cst$-algebras.

We use these terminologies in order to generalize a characterization of continuous Rokhlin property in terms of the dual action of $\dG$ on $G \ltimes A$ as in Lemma 3.8 of \cite{MR2053753} for general compact groups.

\begin{defn}
We say that an action $\alpha$ of $\dG$ on $A$ is \emph{asymptotically representable} if there is a strictly continuous path of (asymptotic) unitaries $u_t \in M(C^*_\lambda(G) \otimes M(A)^\alpha)$ which satisfies 
\begin{align*}
(\dD \otimes \id)(u_t) - (u_t)_{13} (u_t)_{23} \to &0,\\
\alpha(a) - u_t^* (1 \otimes a) u_t \to &0
\end{align*}
for any $a \in A$ strictly as $t \to \infty$.
\end{defn}

\begin{lem}\label{lem:asymprep}
Let $G$ be a compact group.
\begin{enumerate}
\item Let $(A,\alpha)$ be a $G$-$\Cst$-algebra. Then $\alpha$ has continuous Rokhlin property if and only if the dual action $\hat \alpha$ of $\hat{G}$ on $G \ltimes A$ is asymptotically representable.
\item Let $(A,\alpha)$ be a $\dG$-$\Cst$-algebra. Then $\alpha$ is asymptotically representable if and only if the dual action $\hat \alpha$ of $G$ on $\dG \ltimes A$ has continuous Rokhlin property.
\end{enumerate}
\end{lem}
\begin{proof}
For (1), first assume $\alpha$ has continuous Rokhlin property with the averaging map $\psi_t$.
Set
\[u_t := (\id \otimes \psi_t)(W \otimes 1_A) \in M(C^*_\lambda(G) \otimes A).\]
Since $\psi_t$ is $G$-equivariant and asymptotically central in $A$, we obtain $\hat{\alpha} (a) - u_t^* (1 \otimes a) u_t \to 0$. 

Conversely assume $\hat \alpha$ is asymptotically representable with the implementing asymptotic representation $u_t \in M(C^*_\lambda(G) \otimes A)$. Then $u_t$ determines a unitary representation $u \in M(C^*_\lambda(G) \otimes \mathfrak{A}A)$, hence gives rise to a unital $*$-homomorphism
\[\psi: C(G) \to \mathfrak{A}A.\]
Uniform continuity, $G$-equivariance and asymptotic centrality follow from the relation $\alpha(a) - u_t^*(1 \otimes a) u_t \to 0$.

The assertions in (2) follow from (1) and the Baaj-Skandalis duality. 
\end{proof}

\subsection{The model action}
Next, we construct a continuous Rokhlin action of a Hodgkin Lie group $G$ on a Kirchberg algebra which is $\KK^G$-equivalent to $C(G)$, more precisely, an asymptotically representable action of $\dG$ on $\cO_\infty$ whose crossed product is purely infinite. 
Our method is close to that of \cite{mathOA14061208}.

Pick a finite sequence of elements $S := (g_1, g_2, \ldots g_n)$ in $G$. Then this induces a finite dimensional representation
\[\pi_S : C(G) \to \mathbb{M}_n, \ \pi_S(f) = {\rm diag}(f(g_i))\]
and the corresponding representation $U_S := {\rm diag}(\lambda_{g_i}) \in M(C^*_\lambda(G) \otimes \mathbb{M}_n)$ of $\dG$.
Observe that
\[U_S \otimes U_T := (U_T)_{13} (U_S)_{12} \simeq U_{S \otimes T}\]
where $S \otimes T := \{hg \mid g \in S, h \in T\}$.
Hereafter, we fix symmetric $S = (g_1,g_2, \ldots, g_n)$ which topologically generates $G$ and contains $1$.

We prepare the following elementary lemma which plays an essential role in the proof of purely infiniteness of the crossed product.
\begin{lem}\label{Lemma}
For any $\varepsilon > 0$ and any finite subset $1 \not \in F \subset \Irr (G)$, there exists $k \in {\mathbb Z}_+$ and a nonzero projection $p \in \mathbb{M}_n^{\otimes k}$ such that
\[\| (\pi \otimes \id )((1_{C^*_\lambda(G)} \otimes p)(U_S^{\otimes k})^* (1_{C^*_\lambda(G)} \otimes p)U_S^{\otimes k}) \| _{\Bop (\Hilb _\pi) \otimes \mathbb{M}_n^{\otimes k}} < \varepsilon\]
for any $\pi \in F$.
\end{lem}
\begin{proof}
Since $\displaystyle \bigcup_k S^{\otimes k}$ is dense in $G$, we have a net of vectors $\xi_i \in \ell^2(S^{\otimes {k_i}})$ such that 
\[(\pi_S(f) \xi_i, \xi_i) \to \int_G f(g) dg\]
for any $f \in C(G)$.
Hence there exists $k \in \zahl _{>0}$ and a unit vector $\xi \in \ell^2(S^{\otimes k})$ such that
\[\| \langle (\id_{\Bop (\Hilb _\pi)} \otimes \pi_S^{\otimes k})(U_\pi)(1_{\Bop (\Hilb _\pi)} \otimes \xi),  1_{\Bop (\Hilb _\pi)} \otimes \xi \rangle \| < \varepsilon\]
for any $\pi \in F$, where $U_\pi \in \Bop(\Hilb_\pi)\otimes C(G)$ is the unitary representation of $G$ corresponding to $\pi$.
Now we take the projection $p \in \mathbb{M}_n^{\otimes k}$ onto $\mathbb C \xi$.
Using
\[(\id _{\Bop (\Hilb _\pi )} \otimes \pi_S^{\otimes k})(U_\pi) = (\pi \otimes \id _{\mathbb{M}_n^{\otimes k}})(U_S^{\otimes k}),\]
we get $p$ satisfies the desired property as follows:
\begin{align*}
& \| (\pi \otimes \id )((1 \otimes p)(U_S^{\otimes k})^* (1 \otimes p)U_S^{\otimes k}) \|_{\Bop (\Hilb _\pi) \otimes \mathbb{M}_n^{\otimes k}} \\
= & \| \langle (\pi \otimes \id)(U_S^{\otimes k})(1 \otimes \xi), 1 \otimes \xi \rangle ^* (1 \otimes p) (\pi \otimes \id)(U_S^{\otimes k})  \|_{\Bop (\Hilb _\pi) \otimes \mathbb{M}_n^{\otimes k}} <  \varepsilon.
\end{align*}
\end{proof}
Consider a unital embedding
\[\mathbb{M}_n \oplus \mathbb{M}_{n+1} \to \cO_\infty\]
and denote the canonical image of $U_S \oplus 1$ by $V$.
This is a unitary representation of $\dG$ on $\cO_\infty$.
Let $\displaystyle A := \bigotimes \nolimits_{k \in {\mathbb Z}_+} \cO_\infty \simeq \cO_\infty$ and consider the following action $\alpha$ of $\dG$ on $A$ given by
\[\alpha(x) = \lim_{k \to \infty} (V^{\otimes k})^* (1 \otimes x) V^{\otimes k}.\]
Here the limit makes sense because the sequence is eventually constant for any $x \in (\cO_\infty)^{\otimes l}$. Namely, for any $k \geq l$ we have
\[(V^{\otimes k})^* (1 \otimes x) V^{\otimes k} = (V^{\otimes l})^* (1 \otimes x) V^{\otimes l}.\]

\begin{thm}\label{thm:model}
Let $G$ be a Hodgkin Lie group. The crossed product $\cO(G):=\dG _\alpha \ltimes A$ with the dual action of $G$ is a Kirchberg $G$-algebra with continuous Rokhlin property which is $\KK^G$-equivalent to $C(G)$ by the canonical unital inclusion $C(G) = C^*(\dG) \to \dG {\, }_{\alpha}\! \ltimes A$.
\end{thm}
Asymptotic representability of $\alpha$, which implies continuous Rokhlin property of $\hat{\alpha}$, follows from connectedness of $G$ since we have a homotopy of diagonal unitary representations $U_t$ connecting $U_S$ and $1$. 
The desired $\KK^G$-equivalence follows from Proposition \ref{prp:BC} and the Baaj-Skandalis duality.
Separability and nuclearity of $\cO(G)$ also follow by definition because $G$ is second countable and coamenable. 
Hence the rest part of the proof is to show that $\cO(G)$ is purely infinite. 
For this, we follow the argument appeared in \citelist{\cite{MR634163}\cite{MR1424962}}.
\begin{lem}
For any $\varepsilon > 0$, finite subset $X \subset A$ and $1 \not \in F \subset {\rm Irr}(G)$,
we have $x \in A_+$ such that for any $z \in X$
\begin{align*}
&\| xzx \| \geq (1 - \varepsilon) \|z\|, \\
&\| (1 \otimes xz)\alpha_\pi(x) \| < \varepsilon \| z \|.
\end{align*}
\end{lem}
\begin{proof}
We may assume $X \subset (\cO_\infty)^{\otimes k}$. Pick $p \in (\mathbb{M}_n)^{\otimes l}$ as in Lemma \ref{Lemma}. Consider the inclusion
\[(\mathbb{M}_n)^{\otimes l} \subset (\mathbb{M}_n \oplus \mathbb{M}_{n+1})^{\otimes l} \subset (\cO_\infty)^{\otimes l}\]
and let us denote the image of $p$ by $q$. Then $q$ is a nonzero projection such that
\[\| (1_{\Bop (\Hilb _\pi)} \otimes q) \alpha_\pi(q) \| < \varepsilon\]
for any $\pi \in F$.
Now,
\[x := 1_{\cO_\infty^{\otimes k}} \otimes q \in (\cO_\infty)^{\otimes (k+l)}\]
satisfies the property. In fact, since $x$ commutes with any $z \in X$, the first condition is obvious and the second one follows from
\[(1 \otimes xz)\alpha_\pi(x) = (1 \otimes z) (V^{\otimes k})^* ((1 \otimes q)\alpha_\pi(q))_{1(k+2)(k+3)\ldots(k+l+1)} V^{\otimes k}.\]
\end{proof}
\begin{cor}
The $\Cst$-algebra $\dG {\, }_{\alpha}\! \ltimes A$ is purely infinite.
\end{cor}
\begin{proof}
For any positive $a \in G {\, }_{\alpha}\! \ltimes A $, we approximate $a$ by positive
\[c = \sum_{\pi,i,j} c^\pi_{ij} u^\pi_{ij} \in \dG {\, }_{\alpha, \mathrm{alg}}\! \ltimes A \]
such that $\| a - c \| < \varepsilon$, where $u_{ij}^\pi$ are matrix coefficients of $\pi$ with respect to a fixed basis and $\dG {\, }_{\alpha, \mathrm{alg}}\! \ltimes A$ is the $*$-subalgebra of $\dG \ltimes A$ generated by the polynomial algebra $\mathrm{Pol}(G)$ of $G$ and $A$.

We may assume $\| c^1 \| = 1$.
Let $F \subset {\rm Irr}(G)$ be the support of $c$, $n$ the number of nonzero $c^\pi_{ij}$'s.
Applying the previous lemma for $X = \{c^\pi_{ij} \}$ to get $x \in A_+$ such that
\begin{align*}
&\| x c^1 x \| \geq 1 - \varepsilon , \\
&\| (1 \otimes x c^\pi_{ij}) \alpha_\pi(x) \| < \varepsilon / n.
\end{align*}
Hence we get $\|x c x - x c^1 x \| < 1 - 2 \varepsilon$. The rest of the argument is completely the same as Lemma 10 in \cite{MR1424962}.
\end{proof}

\subsection{Classification}
An important feature of Kirchberg algebras is the Kirchberg-Phillips classification: two Kirchberg algebras are isomorphic if and only if there is a $\KK$-equivalence between them preserving unit classes (Theorem 4.2.4 of \cite{MR1745197}). It is generalized for equivariant setting when the actions have Rokhlin property.

\begin{lem}\label{lem:intw}
Let $A$ and $B$ be $G$-$\Cst$-algebras with $G$-actions $\alpha$ and $\beta$.
Assume $B$ has Rokhlin property and we have a $*$-homomorphism $\varphi: A \to B$ such that there exists a sequence of unitaries $(u_n)_{n \in \mathbb{N}} \in B \otimes C(G)$ such that
\[\Ad(u_n(g)) \circ \varphi(x)\to \beta_g \circ \varphi \circ \alpha_g^{-1} (x) \text{ as } n \to \infty\]
Then we have a $G$-equivariant $\ast$-homomorphism $\psi: A \to B$ which is approximately unitarily equivalent to $\varphi$.
Moreover if $\varphi$ is an isomorphism, then $\psi$ can be taken to be an isomorphism.
\end{lem}
\begin{proof}
Put $\theta: A \to B \otimes C(G)$ as $\theta(x)(g) := \beta_g \circ \varphi \circ \alpha_g^{-1}(x)$. Then $\theta$ is a $G$-equivariant $\ast$-homomorphism.	Fix a $G$-invariant compact set $F \subset A$ and $\varepsilon > 0$.
By assumption, we may take a unitary $u \in B \otimes C(G)$ such that
\[\|u(\varphi(x) \otimes 1) u^* - \theta(x) \| < \varepsilon\]
for $x \in F$.
Now we take the Rokhlin approximation $\chi: B \otimes C(G) \to B$ such that
\begin{itemize}
\item $\| \chi(u)^* \chi(u) - 1\| < \varepsilon$, $\| \chi(u) \chi(u)^* - 1\| < \varepsilon,$
\item $\| \chi(u) \chi(\varphi(x)\otimes 1) \chi(u)^* - \chi \circ \theta(x) \| < \varepsilon$ for $x \in F$,
\item $\| \chi(\varphi(x) \otimes 1) - \varphi(x) \| < \varepsilon$ for $x \in F$.
\end{itemize}
Take the unitary $v := \chi(u) |\chi(u)|^{-1}$. Then we have $\| \chi(u) - v \| < \varepsilon/2$ and hence
\[\| v \varphi(x) v^* - \chi \circ \theta(x) \| < \frac{3}{2}\varepsilon.\]
Since $\chi \circ \theta(x)$ is $G$-equivariant, we get
\[\| v \varphi(\alpha_g(x)) v^* - \beta_g(v \varphi(x) v^*)\| < 3 \varepsilon.\]Moreover since
\[\|\theta(x) - \varphi(x) \otimes 1 \| = \sup_{g \in G} \| \beta_g \circ \varphi \circ \alpha_g^{-1}(x) - x \|,\]
we get
\[\| \chi \circ \theta(x) - \varphi(x) \| < \sup_{g \in G} \| \beta_g \circ \varphi \circ \alpha_g^{-1}(x) - x \| + \varepsilon.\]
Therefore, the intertwining argument in Theorem 3.5 of \cite{MR2053753} and Lemma 5.1 of \cite{MR2047851} works for this situation and we obtain the conclusion.
\end{proof}

\begin{prp}\label{lem:KKconj}
Let $G$ be a compact group and let $A$ and $B$ be Kirchberg $G$-algebras with Rokhlin property. If there is a $\KK^G$-equivalence from $A$ to $B$ mapping $[1_A]$ to $[1_B]$, then they are conjugate.
\end{prp}
\begin{proof}
Let $\varphi ' :q_sA \to B \otimes \Kop _G$ be a $G$-equivariant $\ast$-homomorphism representing a $\KK^G$-equivalence. Since $\varphi'$ is $G$-equivariant, we have $\beta \circ \varphi '=(\varphi' \otimes \id _{C(G)}) \circ \alpha$, which implies 
\[[\beta]=([\varphi '] \otimes \id _{C(G)})\circ [\alpha ] \circ [\varphi ']^{-1} \in \KK (B,B \otimes C(G)).\]
Now, choose a $\ast$-isomorphism $\varphi:A\to B$ such that $[\varphi]=\Res _G[\varphi'] \in \KK(A,B)$ (it is possible thanks to Corollary 4.2.2 of \cite{MR1745197}). Then we obtain two $\ast$-homomorphisms $(\varphi \otimes \id _{C(G)}) \circ \alpha \circ \varphi^{-1}$ and $\beta$, which determines the same element in $\KK(B,B\otimes C(G))$. Thanks to Theorem 4.1.1 of \cite{MR1745197}, they are asymptotically unitarily equivalent. By Lemma \ref{lem:intw}, two $G$-actions $(\varphi \circ \alpha _g \circ \varphi ^{-1})_g$ and $\beta $ on $B$ are conjugate.
\end{proof}

\begin{thm}\label{thm:Kir}
Let $G$ be a Hodgkin Lie group.
\begin{itemize}
\item A Kirchberg $G$-algebra $A$ with continuous Rokhlin property is $G$-equivariantly isomorphic to $A^\alpha \otimes \cO(G)$.
\item Two Kirchberg $G$-algebras $A$ and $B$ with continuous Rokhlin property are isomorphic if and only if the fixed point algebras $A^\alpha$ and $B^\beta$ are isomorphic. When $A$ and $B$ are UCT-Kirchberg $G$-algebras, $A \cong B$ if and only if $(\K_0(A^\alpha), [1_{A^\alpha}], \K_1(A^\alpha )) \cong (\K_0(B^\beta ), [1_{B^\beta}], \K_1(B^\beta))$. 
\item A UCT-Kirchberg algebra $A$ in the Cuntz standard form (i.e.\ $[1_A]=0 \in K_0(A)$) admits a $G$-action with continuous Rokhlin property if and only if there is a countable abelian group $M$ such that $K_*(A)$ is isomorphic to $M^{\oplus n}$ where $n=2^{\rank G-1}$. In this case, $M\cong \K_0(A^\alpha) \oplus \K_1(A ^\alpha)$.
\end{itemize}
\end{thm}
\begin{proof}
By Theorem \ref{thm:KKHod}, for every $G$-$\Cst$-algebra $A$ with continuous Rokhlin property, there is a $\KK^G$-equivalence from $A^\alpha \otimes \cO(G)$ to $A$ preserving unit elements. Moreover, as is shown in Theorem \ref{thm:model}, both of  $A^\alpha \otimes \cO(G)$ and $A$ has continuous Rokhlin property. Since the fixed point algebra $A^\alpha$ of a (continuous) Rokhlin action on Kirchberg algebras is again a Kirchberg algebra (Corollary 3.20 of \cite{mathOA14081946}), $A$ and $A^\alpha \otimes \cO(G)$ are $G$-equivariantly isomorphic by Proposition \ref{lem:KKconj}. 

The second assertion follows from Theorem 4.2.4 of \cite{MR1745197}, Corollary \ref{cor:classify} and Proposition \ref{lem:KKconj}. The third assertion follows from the $\KK$-equivalence $C(G) \sim \comp ^n \oplus C_0(\real)^n$ as in the proof of Corolalry \ref{cor:Krid}.
\end{proof}

\subsection{Rokhlin property vs.\ continuous Rokhlin property}\label{section:5.3}
We conclude the article by a comparison of Rokhlin and continuous Rokhlin properties. 
\begin{prp}\label{prp:injcRok}
Let $G$ be a compact group and let $A$ be a UCT-Kirchberg $G$-algebra with Rokhlin property. Then, $A$ has continuous Rokhlin property if (and only if) it is $\J_G$-injective. 
\end{prp}
\begin{proof}
We know we have $\xi \in KK^G(A \otimes C(G), A)$ such that $\xi \circ [\iota] = \id _A$. Our goal is to prove that there exists a $G$-equivariant asymptotic morphism $\psi_t: A \otimes C(G) \to A$ such that $\psi_t \circ \iota$ is asymptotically equal to $\id _A$, so that $f \mapsto \psi_t(1 \otimes f)$ gives the desired continuous Rokhlin approximation.

First, we start with taking a $G$-equivariant unital $\ast$-homomorphism $\theta: A \otimes D \to A$ such that $[\theta \circ \iota] = \id_A$, where $D$ is a Kirchberg $G$-algebra with a $G$-equivariant unital $\ast$-homomorphism $i:C(G) \to D$ which induces a $\KK^G$-equivalence (for example, we obtain it by the Cuntz-Pimsner construction \citelist{\cite{MR1426840}\cite{MR2109936}} for the Hilbert $C(G)$-bimodule $L^2(G)^\infty \otimes C(G)$ together with the natural $G$-action).

Let $\phi: A \otimes D \to A$ be a unital $*$-homomorphism with $[\phi]=\Res _G \xi$ (note that $\xi_*[1_{A \otimes D}]=\xi_* i _* [1_A]=[1_A]$). 
Then, we have $[\phi \circ \beta]=[\alpha \circ \phi] \in \KK(A \otimes D, A \otimes C(G))$ where we write $\beta$ for the $G$-action on $A \otimes D$. 

By Lemma \ref{lem:intw}, we obtain a $G$-equivariant unital $\ast$-homomorphism $\phi' : A \otimes D \to A$ which is approximately unitarily equivalent to $\phi$. 
Since $[\phi'] \circ \iota \in \KK^G(A,A)$ induces the identity on $\K_*(A)$, it is a $\KK^G$-equivalence by Corollary \ref{cor:equiv}. 

Let $\sigma : A \otimes D \to A$ be a unital $\ast$-homomorphism representing $([\phi'] \circ \iota)^{-1} \otimes \id_{D}$ under the isomorphism $\KK(A \otimes D,A)\cong \KK^G(A \otimes D, A \otimes D)$. 
Then, the unital $G$-equivariant $\ast$-homomorphism 
\[ \tilde{\sigma}: A \otimes D \to A \otimes C(G), \ \tilde{\sigma}(x)(g):=\alpha _g(\sigma (x)) \]
satisfies $[i \circ \tilde{\sigma}]=([\phi'] \circ \iota)^{-1} \otimes \id_{D}$. 
Now, $\theta :=\phi ' \circ i \circ \tilde{\sigma}$ is the desired $G$-equivariant unital $\ast$-homomorphism because $(([\phi'] \circ \iota)^{-1} \otimes \id_{D})\circ \iota=\iota \circ ([\phi'] \circ \iota)^{-1}$.

By Theorem 4.1.1 of \cite{MR1745197}, we get a path of unitaries $(u_t) \in A$ such that
\[u_t \theta(a \otimes 1) u_t^* \to a\]
for any $a \in A$.	
By an inductive reparametrization, we may assume for any $n \in \mathbb{N}$,
\[\| u_t \theta(a \otimes 1) u_t^* - a \| < 2^{-n} \| a \|\]
for any $t \geq n$ and $a \in F_n$, where $F_n$ is an increasing sequence of self-adjoint compact $G$-invariant subset of $A$ which satisfies
\begin{itemize}
\item $A = \overline{\bigcup_n F_n}$,
\item $A \otimes D = \overline{\bigcup_n \theta^{-1}(F_n)}$ and
\item $\{u_s \mid s \leq n-1\} \subset F_n$.
\end{itemize}
Note that since $F_n$ is $G$-invariant, we also get
\begin{itemize}
\item $\|\alpha(u_t) (\theta(a \otimes 1) \otimes 1) \alpha(u_t)^* - a \otimes 1 \|  < 2^{-n} \| a \|$ for $a \in F_n$,
\item $\|\Ad((u_{n+1} \otimes 1)\alpha(u_{n+1})^*) \alpha(u_n) - \alpha(u_n) \|  < 2^{-n}$.
\end{itemize}
Thanks to the equations above, again by induction, we may take Rokhlin averaging maps $\chi_k: A \otimes C(G) \to A$ such that
\begin{enumerate}
\item $\|\chi_n(\alpha(u_t)) \theta(a \otimes 1) \chi_n(\alpha(u_t))^* - a\| < 2^{-n} \| a \|$,
\item $\|\Ad (u_{n+1} \chi_n(\alpha(u_{n+1}))^*)( \chi_n(\alpha(u_n))) - \chi_n(\alpha(u_n)) \| < 2^{-n}$,
\item $\|\Ad (\chi_n(\alpha(u_n)) \chi_{n-1}(\alpha(u_n))^*)( \chi_{n-1}(\alpha(u_{n-1}))) - \chi_{n-1}(\alpha(u_{n-1})) \| < 2^{-n+1}$,
\item $\| \chi_n(\alpha(u_t))^* \chi_n(\alpha(u_t)) - 1 \|, \| \chi_n(\alpha(u_t))^* \chi_n(\alpha(u_t)) - 1 \| < 2^{-n}$.
\end{enumerate}
for $n \leq t \leq n+1$ and $a \in F_n$ (actually we do not need (2) for the later argument, but we put this to get (3) in the induction).
Then due to the condition (4), for $n \leq t \leq n+1$,
\[v_{n,t} := \chi_n(\alpha(u_t)) |\chi_n(\alpha(u_t))|^{-1}\]
is a $G$-invariant unitary such that
\[\| v_{n,t} - \chi_n(\alpha(u_t)) \| < 2^{-n-1}. \]
Hence we rewrite (1) and (3) as
\begin{enumerate}
\item[(1')] $\|v_{n,t} \theta(a \otimes 1) v_{n,t}^* - a \| < 2^{-n+1} \| a \|$,
\item[(3')] $\|\Ad(v_{n+1,n+1} v_{n,n+1}^*)(v_{n,n}) - v_{n,n} \| < 2^{-n+2}$,
\end{enumerate}
for any $n \leq t \leq n+1$, $a \in F_n$. Note that (1') implies $\| \Ad (v_{n,t}v_{n+1,t}^*)(a) -a \| <2^{-n+2}\ssbk{a}$ for any $a \in F_n$. Now, we construct a desired path
\[
\psi_t(x) := \begin{cases} {\rm Ad}(v_{n+1,n+1} v_{n,n+1}^* v_{n,t/2 - n})\theta(x) & \text{for } t \in [2n,2n+1]\\
\frac{t-2n-1}{2} {\rm Ad}(v_{n+1,n+1} v_{n,n+1}^* v_{n,n}) \theta(x) + \frac{t - 2n - 2}{2} {\rm Ad}(v_{n,n}) \theta(x) & \text{for } t \in [2n+1,2n+2].
\end{cases}
\]
From (1') and (3'), $\psi_t$ is a $G$-equivariant asymptotic morphism since
\ma{
&\| \Ad(v_{n+1,n+1} v_{n,n+1}^* v_{n,n}) \circ \theta(x) - \Ad(v_{n,n}) \circ \theta(x) \|\\
<&\| \Ad (v_{n,n}v_{n+1,n+1} v_{n,n+1}^* ) \circ \theta(x) - \Ad(v_{n,n}) \circ \theta(x) \| +2^{-n+3} < 2^{-n+4}
}
for any $x \in \theta^{-1}(F_n)$. Moreover again from (1') and (3'), $\| \psi_t(a \otimes 1) - a \| \to 0$ as $t \to \infty$, as desired.
\end{proof}

We remark that Proposition \ref{prp:injcRok} holds for general compact groups. In particular, when $G$ is finite, it is related to Izumi's classification of finite group actions on UCT-Kirchberg algebras with Rokhlin property~\cite{MR2047851}. For any finite group $G$, a complete classification of UCT-Kirchberg $G$-algebras with Rokhlin property is given in Corollary 5.4 of \cite{MR2047851} by their $\K_*$-groups as CCT $G$-modules and $[1_A] \in \K _0(A) ^G$. 

In fact, it is an immediate consequence of Theorem \ref{thm:Rokinj} that the $\K_*$-group of $G$-$\Cst$-algebras with continuous Rokhlin property is relatively projective. 
The class of relatively projective modules is strictly smaller than the class of CCT modules although they coincide under some reasonable assumptions. 
Actually, Katsura \cite{RIMS} shows that every CCT modules are given by the third term of a pure exact sequence whose first and second term is relatively projective (cf.\ Proposition \ref{prp:RokcRok}). 
Hence we obtain a UCT-Kirchberg $G$-algebra with Rokhlin property which does not have continuous Rokhlin property. 
Moreover, relative projectivity is also a sufficient condition for continuous Rokhlin property.

\begin{lem}\label{lem:relprojinj}
Let $A$ be the UCT-Kirchberg $G$-algebra such that $\K_*(A)$ are relatively projective $G$-modules. Then, $A$ is $\J_G$-injective.
\end{lem}
\begin{proof}
We write $\cO(G)$ for the model action in Lemma 5.2 of \cite{MR2047851}, which is $\KK ^G$-equivalent to $C(G)$ by the inclusion. Choose an isomorphism $M_* \oplus M_*' \cong N_* \otimes _{\zahl}\zahl [G]$ where $N_*$ are abelian groups. 
Let $A$, $A'$ be the UCT-Kirchberg $G$-algebras in the Cuntz standard form with Rokhlin property corresponding to $M_*$, $M_*'$ respectively and let $B$ be the UCT-Kirchberg algebra in the Cuntz standard form corresponding to $N_*$. 
By Lemma 5.1 of \cite{MR2047851}, we obtain $G$-equivariant $\ast$-homomorphisms $\varphi : A \to B \otimes \cO(G)$ and $\varphi ': A' \to B \otimes \cO(G)$ such that $[\varphi] \oplus [\varphi '] \in \KK^G(A \oplus A', B \otimes \cO(G))$ induces the inclusion of $\K_*$-groups and hence a $\KK ^G$-equivalence by Corollary \ref{cor:equiv}. 
Since $B \otimes \cO(G)$ is $\J_G$-injective, so are direct summands $A$ and $A'$.
\end{proof}

\begin{cor}
Under the one-to-one correspondence given in Corollary 5.4 of \cite{MR2047851}, any triplet $(M_0,x, M_1)$ such that $M_0$ and $M_1$ are relatively projective corresponds to a UCT-Kirchberg $G$-algebra with continuous Rokhlin property. 
\end{cor}
\begin{proof}
It follows from Lemma \ref{lem:relprojinj} and Proposition \ref{prp:injcRok}.
\end{proof}

\begin{cor}
Let $A$ be a UCT-Kirchberg $G$-algebra with Rokhlin property. If both $\K_0(A)$ and $\K_1(A)$ are either finitely generated groups or bounded $p$-groups, then $A$ has continuous Rokhlin property.
\end{cor}
\begin{proof}
It follows from the above corollary and Lemma 3.12 and Lemma 3.13 of \cite{MR2047851}.
\end{proof}

\subsection*{Acknowledgement}
The authors would like to thank their supervisor Yasuyuki Kawahigashi for his support and encouragement. They also would like to thank Takeshi Katsura and Yuhei Suzuki for helpful comments. 
Serious errors in the previous version of this paper is pointed out by Eusebio Gardella. The authors would like to thank him for his helpful comments. 
The first author wishes to thank the KU Leuven, where the last part of this paper was written, for the invitation and hospitality. This work was supported by Research Fellow of the JSPS (No.\ 26-2598 and No.\ 26-7081) and the Program for Leading Graduate Schools, MEXT, Japan.

\bibliographystyle{alpha}
\bibliography{bibABC,bibDEFG,bibHIJK,bibLMN,bibOPQR,bibSTUV,bibWXYZ,arXiv,AtiyahSegal}
\end{document}